\documentclass[10pt,draft,reqno]{amsart}
     \makeatletter
     \def\section{\@startsection{section}{1}%
     \z@{.7\linespacing\@plus\linespacing}{.5\linespacing}%
     {\bfseries
     \centering
     }}
     \def\@secnumfont{\bfseries}
     \makeatother
\usepackage{amsmath,amssymb, amsbsy}
\usepackage{subfigure}
\usepackage{graphpap,latexsym,epsf}
\usepackage{color,psfrag}
\usepackage[dvips]{graphicx}
\usepackage{enumerate}
\usepackage{bbm}
\usepackage{relsize}

\newcommand{\F}{{\mathcal F}}
\newcommand{\ieq}{\begin{equation}}
\newcommand{\eeq}{\end{equation}}
\newcommand{\ieqa}{\begin{eqnarray}}
\newcommand{\eeqa}{\end{eqnarray}}
\newcommand{\ieqas}{\begin{eqnarray*}}
\newcommand{\eeqas}{\end{eqnarray*}}
\newcommand{\f}{\hat{f}}

\newcommand{\1}{\mathlarger{\mathlarger{\mathbbm{1}}}}

\newtheorem{theorem}{Theorem}[section]

\newtheorem{proposition}[theorem]{Proposition}
\newtheorem{corollary}[theorem]{Corollary}
\theoremstyle{definition}
\newtheorem{definition}[theorem]{Definition}

\theoremstyle{remark}

\numberwithin{equation}{section}
\begin{document}

\title[Anticipating stochastic integrals in financial modeling]{A triple comparison between anticipating stochastic integrals in financial modeling}

\author{Joan C. Bastons}
\address{Joan C. Bastons: Departmento de Matem\'aticas, Universidad Aut\'onoma de Madrid, Madrid, 28049, Spain}
\email{joan.bastons@estudiante.uam.es}

\author[Carlos Escudero]{Carlos Escudero*}
\thanks{* This work has been partially supported by the Government of Spain (Ministry of Economy, Industry, and Competitiveness) through Project MTM2015-72907-EXP}
\address{Carlos Escudero: Departmento de Matem\'aticas, Universidad Aut\'onoma de Madrid, Madrid, 28049, Spain}
\email{carlos.escudero@uam.es}



\begin{abstract}
We consider a simplified version of the problem of insider trading in a financial market. We approach it by means of anticipating
stochastic calculus and compare the use of the Hitsuda-Skorokhod, the Ayed-Kuo, and the Russo-Vallois forward integrals within this context.
Our results give some indication that, while the forward integral yields results with a suitable financial meaning, the Hitsuda-Skorokhod and the Ayed-Kuo integrals
do not provide an appropriate formulation of this problem. Further results regarding the use of the Ayed-Kuo integral in this context are
also provided, including the proof of the fact that the expectation of a Russo-Vallois solution is strictly greater than that of an Ayed-Kuo solution.
Finally, we conjecture the explicit solution of an Ayed-Kuo stochastic differential equation that possesses discontinuous
sample paths with finite probability.
\end{abstract}

\maketitle


\section{Introduction}

Stochastic differential equations of the sort
\begin{equation}\label{white}
\frac{dx}{dt}= a(x,t) + b(x,t) \, \xi(t),
\end{equation}
where $\xi(t)$ is a ``white noise'', are ubiquitous mathematical models that arise in different sciences~\cite{hl}.
As written, this model is intuitive but imprecise, and it may become a real mathematical model only if a precise
meaning of the white noise process is given. Classically this has been done through the introduction of the It\^o integral
and its corresponding stochastic differential equation
\begin{equation}\label{ito}
dx = a(x,t) \, dt + b(x,t) \, dB_t,
\end{equation}
where $B_t$ is a Brownian motion, or the Stratonovich integral
\begin{equation}\label{str}
dx = a(x,t) \, dt + b(x,t) \circ dB_t,
\end{equation}
or a different stochastic integral that results from choosing a different evaluation point for the integrand~\cite{mmcc}.
The mathematical theory of existence, uniqueness, and regularity for the solutions to stochastic differential equations of It\^{o}
and Stratonovich type has been built and its applications have been explored~\cite{kuo2006,oksendal}.
One of the most evident differences between equations~\eqref{ito} and~\eqref{str} is their diverse dynamical behavior,
which may affect their stability properties~\cite{tcxh}, and this has direct implications in their applications~\cite{hl}.
Therefore choosing an \emph{interpretation of noise} is the key modeling procedure that takes~\eqref{white} into one
of its precise versions such as~\eqref{ito} or~\eqref{str}. Being a crucial step in applications, and given its model-dependent nature,
this question has generated a vast literature~\cite{mmcc}. While the discussion has been focused on stochastic integrals that
possess adapted integrands, there is nothing genuinely different between this case and that of anticipating stochastic calculus.
In this work we explore the role of noise interpretation in the latter setting.

Anticipating stochastic calculus arises naturally in the study of financial markets when privileged information is considered.
Herein we consider two investment possibilities, one risk-free asset, the bond:
\begin{eqnarray}\nonumber
d S_0 &=& \rho \, S_0 \, d t, \\ \nonumber
S_0(0) &=& M_0,
\end{eqnarray}
and a risky asset, the stock:
\begin{eqnarray}\nonumber
d S_1 &=& \mu \, S_1 \, d t + \sigma \, S_1 \, d B_t, \\ \nonumber
S_1(0) &=& M_1,
\end{eqnarray}
where $M_0, M_1, \rho, \mu, \sigma \in \mathbb{R}^+ := ]0,\infty[$ have the following meaning:
\begin{itemize}
\item $M_0$ is the initial wealth to be invested in the bond,
\item $M_1$ is the initial wealth to be invested in the stock,
\item $\rho$ is the interest rate of the bond,
\item $\mu$ is the appreciation rate of the stock,
\item $\sigma$ is the volatility of the stock.
\end{itemize}
The total initial wealth is given by $M = M_0+M_1$ and we employ the assumption $\mu > \rho$ that expresses the higher
expected return of the risky investment. We also assume that the total initial wealth $M$ is fixed and
the investor freely chooses what part of it is invested in each asset, that is, freely chooses the values of $M_0$ and $M_1$.
Obviously, the total time-dependent wealth is
$$
S^{\text{(I)}}(t)=S_0(t)+S_1(t).
$$
We will consider the dynamics restricted to the time interval $[0,T]$ where the future time $T>0$ is fixed.
We can use It\^o calculus to compute the expectation of the total wealth at time $T$, which is
\begin{equation}\nonumber
\mathbb{E}\left[S^{\text{(I)}}(T)\right] = M_0 \, e^{\rho T} + M_1 \, e^{\mu T}.
\end{equation}
The maximization of this expectation clearly leads to the choice
$$
M_0 = 0, \qquad M_1 = M,
$$
which subsequently leads to the maximal expected wealth
\begin{equation}\label{average}
\mathbb{E}\left[S^{\text{(I)}}(T)\right] = M e^{\mu T}.
\end{equation}
This maximization problem can be considered as a simplified version of the Merton portfolio optimization problem~\cite{merton}.
For instance, the present trader is assumed to be risk-neutral rather than risk-averse. The mathematical simplification
allows for the treatment of the generalization that will be discussed in the following section. This strategy would be the one
chosen by a risk-neutral honest trader. If the trader were a dishonest insider with privileged information on the future price of
the stock, and if this information were used initially to take advantage of it, then the selected strategy would be in general different
from this one. The corresponding discussion as well as the pitfalls to be encountered in the mathematical formulation of this problem
will be presented in the following section.

\section{A simplified model of insider trading}

Assume now that our trader is an insider who has some privileged information on the future value of the stock.
Specifically this insider knows at the initial time $t=0$ the value $B_T$, and therefore the value $S(T)$, but
however s/he does not fully trust this information.
To take advantage of this knowledge s/he uses a modulation of this information in the initial condition. In particular we assume the strategy
\begin{subequations}
\begin{eqnarray}\label{s1}
d S_1 &=& \mu \, S_1 \, d t + \sigma \, S_1 \, d B_t \\ \label{s10}
S_1(0) &=& M \left[ 1 + \frac{\sigma B_T -\sigma^2 T/2}{2(\mu - \rho) T} \right],
\end{eqnarray}
\end{subequations}
and
\begin{subequations}
\begin{eqnarray}\label{rode1}
d S_0 &=& \mu \, S_0 \, d t \\ \label{rode2}
S_0(0) &=& M \frac{\sigma^2 T/2 - \sigma B_T}{2(\mu - \rho) T}.
\end{eqnarray}
\end{subequations}
This strategy is linear in $B_T$ for analytical tractability but also as a way to introduce a certain degree of trust of the insider in the privileged
information s/he possesses. Precisely, the strategy imposes
the same initial investment in the bond and the stock if their values at time $T$ are equal for a given common initial investment.
Additionally, it imposes a null investment in the bond whenever the realization of the Brownian motion yields the average result~\eqref{average}.
Finally, note that we allow negative values for the investments, which means the trader is allowed to borrow money.

The problem we have just described, as formulated in~\eqref{s1}-\eqref{s10}, is ill-posed. While problem~\eqref{rode1}-\eqref{rode2}
can be regarded as an ordinary differential equation subject to a random initial condition, the anticipating character of the initial condition~\eqref{s10}
makes the stochastic differential equation~\eqref{s1} ill-defined in the sense of It\^o. Note however that this problem can be well-posed by changing
the notion of stochastic integration from that of It\^o to a different one that admits anticipating integrands.
The following sections analyze three different choices, the Hitsuda-Skorokhod, the Russo-Vallois, and the Ayed-Kuo stochastic integrals,
and compare the corresponding results. While any of these anticipating stochastic integrals guarantees the well-posedness of the problem under study,
the financial consequences of each choice might be very different.
At this moment it is important to remark that this sort of problem has been approached before through the use of the Russo-Vallois forward integral~\cite{bo,noep,do1,do2,do3,leon,nualart}. On one hand, the justification of such a choice has frequently been made based on technical arguments.
On the other hand, a financial justification is given in~\cite{bo} by means of a buy-and-hold strategy.
This work presents a direct comparison of the financial consequences of the use of each of these anticipating stochastic integrals.
We anticipate that the forward integral is the only one among these that presents desirable modeling properties, at least in our limited setting.
We also note that the problem of insider trading can be approached by means of different methods~\cite{jyc,pk},
but semimartingale approaches are less general than the use of anticipating stochastic calculus~\cite{noep}.

\section{The Hitsuda-Skorokhod integral}

The Hitsuda-Skorokhod integral is an anticipating stochastic integral that was introduced by Hitsuda~\cite{hitsuda} and Skorokhod~\cite{skorokhod}
by means of different methods.
The following definition makes use of the Wiener-It\^o chaos expansion; background on this topic can be found for instance in~\cite{noep,hoeuz}.

\begin{definition}
Let $X \in L^2([0,T]\times \Omega)$ be a square integrable stochastic process. By the Wiener-It\^o chaos expansion, $X$ can be decomposed into an orthogonal series
$$X(t,\omega) = \sum_{n=0}^{\infty} I_n(f_{n,t})$$
in $L^2(\Omega)$, where $f_{n,t}\in L^2([0,T]^n)$ are symmetric functions for all non-negative integers $n$. Thus, we write
$$ f_{n,t}(t_1,\ldots,t_n)=f_n(t_1,\ldots,t_n,t),$$
which is a function defined on $[0,T]^{n+1}$ and symmetric with respect to the first $n$ variables.
The symmetrization of $f_n(t_1,\ldots,t_n,t_{n+1})$ is given by
\begin{align*} \label{simetrizacion}
&\f_n(t_1,\ldots,t_{n+1})= \\
&\frac{1}{n+1}\left[f_n(t_1,\ldots,t_{n+1})+ f_n(t_{n+1},t_2,\ldots,t_1)+\ldots+f_n(t_1,\ldots,t_{n+1},t_{n})\right],
\end{align*}
because we only need to take into account the permutations which exchange the last variable with any other one.
Then, the Hitsuda-Skorokhod integral of $X$ is defined by
\begin{equation*}
\int_0^T X(t,\omega) \, \delta B(t) := \sum_{n=0}^{\infty} I_{n+1}(\f_n),
\end{equation*}
provided that the series converges in $L^2(\Omega)$.
\end{definition}

Following the notation in~\cite{noep}, for the Hitsuda-Skorokhod integral we arrive at the initial value problem
\begin{subequations}
\begin{eqnarray}\label{sko1}
\delta S_1 &=& \mu \, S_1 \, d t + \sigma \, S_1 \, \delta B_t \\ \label{sko2}
S_1(0) &=& M \left[ 1 + \frac{\sigma B_T -\sigma^2 T/2}{2(\mu - \rho) T} \right],
\end{eqnarray}
\end{subequations}
for a Hitsuda-Skorokhod stochastic differential equation.
The existence and uni\-que\-ness theory for linear stochastic differential equations of Hitsuda-Skorokhod type, which covers the present case,
can be found in~\cite{lssdes}.

We remind the reader that the total wealth of the insider is still given by
$$
S^{\text{(HS)}}(T)=S_0(t)+S_1(t).
$$

\begin{theorem}\label{mainthhs}
The expected value of the total wealth of the insider at time $t=T$ is
\begin{eqnarray}\nonumber
\mathbb{E}\left[S^{\text{(HS)}}(T)\right]= M \left\{ \frac{\sigma^2}{4(\mu - \rho)} \, e^{\rho T}
+ \left[ 1 - \frac{\sigma^2}{4(\mu - \rho)} \right] e^{\mu T} \right\},
\end{eqnarray}
for model~\eqref{rode1}-\eqref{rode2} and~\eqref{sko1}-\eqref{sko2}.
\end{theorem}

\begin{proof}
Using Malliavin calculus techniques~\cite{noep} it is possible to solve problem~\eqref{sko1}-\eqref{sko2} explicitly to find
$$
S_1(t)= M \left[ 1 + \frac{\sigma B_T -\sigma^2 T/2}{2(\mu - \rho) T} \right] \diamond
\exp \left[ \left( \mu -\frac{\sigma^2}{2} \right) t + \sigma B_t \right],
$$
where $\diamond$ denotes the Wick product~\cite{noep}.
Now, using the factorization property of the expectation of a Wick product of random variables,
we find for the expected wealth at the terminal time:
\begin{eqnarray}\nonumber
\mathbb{E}\left[S^{\text{(HS)}}(T)\right] &=& \mathbb{E}[S_0(T)] + \mathbb{E}[S_1(T)] \\ \nonumber
&=& \mathbb{E}\left[ M \frac{\sigma^2 T/2 -\sigma B_T}{2(\mu - \rho) T} \right] e^{\rho T} + \\ \nonumber
& &  \mathbb{E} \left\{ M \left[ 1 + \frac{\sigma B_T -\sigma^2 T/2}{2(\mu - \rho) T} \right] \diamond
\exp \left[ \left( \mu -\frac{\sigma^2}{2} \right) T + \sigma B_T \right] \right\} \\ \nonumber
&=& M \frac{\sigma^2 T/2 -\sigma \, \mathbb{E}(B_T)}{2(\mu - \rho) T} \, e^{\rho T} + \\ \nonumber
& &  \mathbb{E} \left\{ M \left[ 1 + \frac{\sigma B_T -\sigma^2 T/2}{2(\mu - \rho) T} \right] \right\}
\mathbb{E} \left\{ \exp \left[ \left( \mu -\frac{\sigma^2}{2} \right) T + \sigma B_T \right] \right\} \\ \nonumber
&=& M \frac{\sigma^2}{4(\mu - \rho)} \, e^{\rho T}
+ M \left[ 1 + \frac{\sigma \mathbb{E}(B_T) -\sigma^2 T/2}{2(\mu - \rho) T} \right] e^{\mu T} \\ \nonumber
&=& M \left\{ \frac{\sigma^2}{4(\mu - \rho)} \, e^{\rho T}
+ \left[ 1 - \frac{\sigma^2}{4(\mu - \rho)} \right] e^{\mu T} \right\},
\end{eqnarray}
where we have also used that $B_T \sim \mathcal{N}(0,T)$.
\end{proof}

\begin{corollary}
The expected value of the total wealth of the insider at time $t=T$ is strictly smaller than that of the honest trader, i.~e.
$$
\mathbb{E}\left[S^{\text{(HS)}}(T)\right] < \mathbb{E}\left[S^{\text{(I)}}(T)\right].
$$
\end{corollary}

\section{The Russo-Vallois integral}

The Russo-Vallois forward integral was introduced in~\cite{russovallois}. As well as the Hitsuda-Skorokhod integral,
it generalizes the It\^o integral and allows to integrate anticipating processes, but in general produces different results from the former~\cite{noep}.

\begin{definition}
A stochastic process $\{\varphi(t), t \in[0,T]\}$ is \textit{forward integrable} (in the weak sense) over $[0,T]$ with respect to Brownian motion $\{B(t), t\in[0,T]\}$ if there exists a stochastic process $\{I(t), t \in [0,T]\}$ such that
\begin{equation}\nonumber
\sup_{t\in[0,T]} \left | \int_0^t \varphi(s) \frac{B(s+\varepsilon)-B(s)}{\varepsilon} ds - I(t) \right| \to 0, \ \ \  \ \mbox{as} \ \varepsilon \to 0^+,
\end{equation}
in probability. In this case, $I(t)$ is the \textit{forward integral} of $\varphi(t)$ with respect to $B(t)$ on $[0,T]$ and we denote
\begin{equation*}
I(t) := \int_0^t \varphi(s) \, d^- B(s),  \ \ \ \ t \in [0,T].
\end{equation*}
\end{definition}

When the choice is the Russo-Vallois integral, we face the initial value problem
\begin{subequations}
\begin{eqnarray}\label{rv1}
d^- S_1 &=& \mu \, S_1 \, dt + \sigma \, S_1 \, d^- B_t \\ \label{rv2}
S_1(0) &=& M \left[ 1 + \frac{\sigma B_T -\sigma^2 T/2}{2(\mu - \rho) T} \right],
\end{eqnarray}
\end{subequations}
for a forward stochastic differential equation. Again, this problem possesses a unique solution~\cite{noep}.

\begin{theorem}\label{mainthrv}
The expected value of the total wealth of the insider at time $t=T$ is
\begin{eqnarray}\nonumber
\mathbb{E}\left[S^{\text{(RV)}}(T)\right]= M \left\{ \frac{\sigma^2}{4(\mu - \rho)} \, e^{\rho T}
+ \left[ 1 + \frac{\sigma^2}{4(\mu - \rho)} \right] e^{\mu T} \right\},
\end{eqnarray}
for model~\eqref{rode1}-\eqref{rode2} and~\eqref{rv1}-\eqref{rv2}.
\end{theorem}

\begin{proof}
The Russo-Vallois integral preserves It\^o calculus~\cite{noep} so using this classical stochastic calculus
it is possible to solve problem~\eqref{rv1}-\eqref{rv2} explicitly to find
$$
S_1(t)= M \left[ 1 + \frac{\sigma B_T -\sigma^2 T/2}{2(\mu - \rho) T} \right]
\exp \left[ \left( \mu -\frac{\sigma^2}{2} \right) t + \sigma B_t \right].
$$
For the expected wealth at the terminal time we find:
\begin{eqnarray}\nonumber
\mathbb{E}\left[S^{\text{(RV)}}(T)\right] &=& \mathbb{E}[S_0(T)] + \mathbb{E}[S_1(T)] \\ \nonumber
&=& \mathbb{E}\left[ M \frac{\sigma^2 T/2 -\sigma B_T}{2(\mu - \rho) T} \right] e^{\rho T} \\ \nonumber
& & + \mathbb{E} \left\{ M \left[ 1 + \frac{\sigma B_T -\sigma^2 T/2}{2(\mu - \rho) T} \right]
\exp \left[ \left( \mu -\frac{\sigma^2}{2} \right) T + \sigma B_T \right] \right\} \\ \nonumber
&=& M \frac{\sigma^2 T/2 -\sigma \, \mathbb{E}(B_T)}{2(\mu - \rho) T} \, e^{\rho T} \\ \nonumber
& & + M \frac{\sigma}{2(\mu - \rho) T} \, \mathbb{E} \left\{ B_T \,
\exp \left[ \left( \mu -\frac{\sigma^2}{2} \right) T + \sigma B_T \right] \right\} \\ \nonumber
& & + M \left[ 1 - \frac{\sigma^2}{4(\mu - \rho)} \right] \mathbb{E} \left\{
\exp \left[ \left( \mu -\frac{\sigma^2}{2} \right) T + \sigma B_T \right] \right\} \\ \nonumber
&=& M \frac{\sigma^2}{4(\mu - \rho)} \, e^{\rho T} \\ \nonumber
& & + M \frac{\sigma^2}{2(\mu - \rho)} \, e^{\mu T} \\ \nonumber
& & + M \left[ 1 - \frac{\sigma^2}{4(\mu - \rho)} \right] e^{\mu T} \\ \nonumber
&=& M \left\{ \frac{\sigma^2}{4(\mu - \rho)} \, e^{\rho T}
+ \left[ 1 + \frac{\sigma^2}{4(\mu - \rho)} \right] e^{\mu T} \right\},
\end{eqnarray}
where we have used that $B_T \sim \mathcal{N}(0,T)$.
\end{proof}

\begin{corollary}
The expected value of the total wealth of the insider at time $t=T$ is strictly larger than that of the honest trader, i.~e.
$$
\mathbb{E}\left[S^{\text{(RV)}}(T)\right] > \mathbb{E}\left[S^{\text{(I)}}(T)\right].
$$
\end{corollary}

\section{The Ayed-Kuo integral}

The Ayed-Kuo integral was introduced in~\cite{akuo1,akuo2} and, as the two previous theories, generalizes the It\^o integral to anticipating integrands.
Let us now consider a Brownian motion $\{B_t, t \geq 0\}$ and a filtration $\{\F_t, t \geq 0\}$ such that:
\begin{itemize}
\item[(i)] For all $t \geq 0$, $B_t$ is $\F_t$-measurable,
\item[(ii)] for all $0 \leq s < t$, $B_t-B_s$ is independent of $\F_s$.
\end{itemize}

We also recall the notion of instantly independent stochastic process introduced in~\cite{akuo1}.

\begin{definition}
A stochastic process $\{\varphi(t)$, $t \in [0,T]\}$ is \textit{instantly independent} with respect to the filtration $\{\F_t, t \in [0,T]\}$ if and only if $\varphi(t)$ is independent of $\F_t$ for each $t \in [0,T]$.
\end{definition}

The Ayed-Kuo integral is defined for integrands that can be expressed as the product of an instantly independent process and a $\{\F_t\}$-adapted stochastic process. Now we recall its definition as given in~\cite{akuo1}.

\begin{definition}
Let $\{f(t), t\in [0,T]\}$ be a $\{\F_t\}$-adapted stochastic process and $\{\varphi(t), t\in [0,T]\}$ an instantly independent stochastic process with respect to $\{\F_t, t \in [0,T]\}$. We define the \textit{Ayed-Kuo stochastic integral} of $f(t)\varphi(t)$ by
\begin{equation}\nonumber
\int_0^T f(t) \varphi(t) \, d^*B_t : = \lim_{\left| \Pi_n \right| \to 0} \sum_{i=1}^n f(t_{i-1}) \varphi(t_i) \left[B(t_i)-B(t_{i-1}) \right]
\end{equation}
in probability, provided that the limit exists, where the $\Pi_n$'s are partitions of the interval $[0,T]$.
\end{definition}

For the Ayed-Kuo integral the initial value problem reads
\begin{subequations}
\begin{eqnarray} \label{ak1}
d^* S_1 &=& \mu \, S_1 \, dt + \sigma \, S_1 \, d^* B_t \\ \label{ak2}
S_1(0) &=& M \left[ 1 + \frac{\sigma B_T -\sigma^2 T/2}{2(\mu - \rho) T} \right],
\end{eqnarray}
\end{subequations}
where $d^*$ denotes the Ayed-Kuo stochastic integral.
This anticipating stochastic differential equation, as the previous two cases, has a unique and explicitly computable solution~\cite{hksz}.

\begin{theorem}\label{mainthak}
The expected value of the total wealth of the insider at time $t=T$ is
\begin{eqnarray}\nonumber
\mathbb{E}\left[S^{\text{(AK)}}(T)\right]= M \left\{ \frac{\sigma^2}{4(\mu - \rho)} \, e^{\rho T}
+ \left[ 1 - \frac{\sigma^2}{4(\mu - \rho)} \right] e^{\mu T} \right\},
\end{eqnarray}
for model~\eqref{rode1}-\eqref{rode2} and~\eqref{ak1}-\eqref{ak2}.
\end{theorem}

\begin{proof}
Using the It\^o formula for the Ayed-Kuo integral~\cite{hksz} it is possible to solve problem~\eqref{ak1}-\eqref{ak2} explicitly to find
\begin{equation}\label{explicitak}
S_1(t)= M \left[ 1 + \frac{\sigma B_T -\sigma^2 t -\sigma^2 T/2}{2(\mu - \rho) T} \right]
\exp \left[ \left( \mu -\frac{\sigma^2}{2} \right) t + \sigma B_t \right].
\end{equation}
For the expected wealth at the terminal time we find:
\begin{eqnarray}\nonumber
\mathbb{E}\left[S^{\text{(AK)}}(T)\right] &=& \mathbb{E}[S_0(T)] + \mathbb{E}[S_1(T)] \\ \nonumber
&=& \mathbb{E}\left[ M \frac{\sigma^2 T/2 -\sigma B_T}{2(\mu - \rho) T} \right] e^{\rho T} \\ \nonumber
& & + \mathbb{E} \left\{ M \left[ 1 + \frac{\sigma B_T -3\sigma^2 T/2}{2(\mu - \rho) T} \right]
\exp \left[ \left( \mu -\frac{\sigma^2}{2} \right) T + \sigma B_T \right] \right\} \\ \nonumber
&=& M \frac{\sigma^2 T/2 -\sigma \, \mathbb{E}(B_T)}{2(\mu - \rho) T} \, e^{\rho T} \\ \nonumber
& & + M \frac{\sigma}{2(\mu - \rho) T} \, \mathbb{E} \left\{ B_T \,
\exp \left[ \left( \mu -\frac{\sigma^2}{2} \right) T + \sigma B_T \right] \right\} \\ \nonumber
& & + M \left[ 1 - \frac{3\sigma^2}{4(\mu - \rho)} \right] \mathbb{E} \left\{
\exp \left[ \left( \mu -\frac{\sigma^2}{2} \right) T + \sigma B_T \right] \right\} \\ \nonumber
&=& M \frac{\sigma^2}{4(\mu - \rho)} \, e^{\rho T} \\ \nonumber
& & + M \frac{\sigma^2}{2(\mu - \rho)} \, e^{\mu T} \\ \nonumber
& & + M \left[ 1 - \frac{3\sigma^2}{4(\mu - \rho)} \right] e^{\mu T} \\ \nonumber
&=& M \left\{ \frac{\sigma^2}{4(\mu - \rho)} \, e^{\rho T}
+ \left[ 1 - \frac{\sigma^2}{4(\mu - \rho)} \right] e^{\mu T} \right\},
\end{eqnarray}
where we have used that $B_T \sim \mathcal{N}(0,T)$.
\end{proof}

\begin{corollary}
The expected value of the total wealth of the insider at time $t=T$ is strictly smaller than that of the honest trader, i.~e.
$$
\mathbb{E}\left[S^{\text{(AK)}}(T)\right] < \mathbb{E}\left[S^{\text{(I)}}(T)\right].
$$
\end{corollary}

The fact that the statements of Theorems~\ref{mainthhs} and~\ref{mainthak} coincides identically is not a casuality.
This follows from the conjectured equivalence
between the Hitsuda-Skorokhod and the Ayed-Kuo integrals~\cite{kuo2014}, which can be proven in this particular case.

\begin{proposition}\label{prophssol}
Formula~\eqref{explicitak} yields the unique solution to initial value problem~\eqref{sko1}-\eqref{sko2}.
\end{proposition}

\begin{proof}
Initial value problem~\eqref{sko1}-\eqref{sko2} is equivalent to
\begin{subequations}
\begin{eqnarray}\label{hs1}
d^- S_1 &=& \left( \mu \, S_1 - \sigma \, D_{t^+} S_1 \right) d t + \sigma \, S_1 \, d^- B_t \\ \label{hs2}
S_1(0) &=& M \left[ 1 + \frac{\sigma B_T -\sigma^2 T/2}{2(\mu - \rho) T} \right],
\end{eqnarray}
\end{subequations}
where $D_{t^+} S_1(t) := \lim_{s \to t^+} D_{s} S_1(t)$ and $D_{s} S_1(t)$ denotes the Malliavin derivative of $S_1(t)$, see~\cite{noep,russovallois}.
Recalling from the proof of Theorem~\ref{mainthhs} that
$$
S_1(t)= M \left[ 1 + \frac{\sigma B_T -\sigma^2 T/2}{2(\mu - \rho) T} \right] \diamond
\exp \left[ \left( \mu -\frac{\sigma^2}{2} \right) t + \sigma B_t \right],
$$
we can compute
\begin{equation}\label{malliavin}
D_{t^+} S_1(t) = \frac{\sigma M}{2(\mu - \rho) T}
\exp \left[ \left( \mu -\frac{\sigma^2}{2} \right) t + \sigma B_t \right].
\end{equation}
Since equation~\eqref{hs1} is a Russo-Vallois forward stochastic differential equation the usual rules of It\^o stochastic calculus apply to it.
Then it is a simple exercise to check that formula~\eqref{explicitak} solves initial value problem~\eqref{hs1}-\eqref{hs2} with
the substitution~\eqref{malliavin}.
\end{proof}

Finally, we state a theorem that shows that the expected value of the wealth of the Ayed-Kuo insider always falls under the expected value
of the wealth of the Russo-Vallois insider.

\begin{theorem}
The respective solutions to the initial value problems
\begin{eqnarray}\nonumber
d^- S^- &=& \mu \, S^- \, dt + \sigma \, S^- \, d^- B_t \\ \nonumber
S^-(0) &=& \mathcal{C}(B_T),
\end{eqnarray}
and
\begin{eqnarray}\nonumber
d^* S^* &=& \mu \, S^* \, dt + \sigma \, S^* \, d^* B_t \\ \nonumber
S^*(0) &=& \mathcal{C}(B_T),
\end{eqnarray}
where $\mathcal{C}(\cdot)$ denotes an arbitrary monotonically increasing function that is both nonconstant and continuous, fulfil
$$
\mathbb{E}[S^*(T)] < \mathbb{E}[S^-(T)].
$$
\end{theorem}

\begin{proof}
We can compute the explicit solutions using respectively the calculus rules for the Russo-Vallois integral~\cite{noep}
and the Ayed-Kuo integral~\cite{hksz} to find
$$
S^-(t)= \mathcal{C}(B_T) \exp \left[ \left( \mu -\frac{\sigma^2}{2} \right) t + \sigma B_t \right],
$$
and
\begin{equation}\label{sstar}
S^*(t)= \mathcal{C}(B_T -\sigma t) \exp \left[ \left( \mu -\frac{\sigma^2}{2} \right) t + \sigma B_t \right].
\end{equation}
Since by monotonicity $\mathcal{C}(B_T -\sigma t) \le \mathcal{C}(B_T)$ for all $t>0$, with the inequality being strict for $B_T$ taking values in
at least some interval of $\mathbb{R}$, we get
\begin{align*}
&\mathbb{E}[S^*(T)] \\
&= \frac{1}{\sqrt{2 \pi T}} \int_{\mathbb{R}} \mathcal{C}(B_T -\sigma T)
\exp \left[ \left( \mu -\frac{\sigma^2}{2} \right) T + \sigma B_T \right] \exp \left( -\frac{B_T^2}{2T} \right)  dB_T \\ \nonumber
&< \frac{1}{\sqrt{2 \pi T}} \int_{\mathbb{R}} \mathcal{C}(B_T)
\exp \left[ \left( \mu -\frac{\sigma^2}{2} \right) T + \sigma B_T \right] \exp \left( -\frac{B_T^2}{2T} \right) \, dB_T \\ \nonumber
&= \mathbb{E}[S^-(T)].
\end{align*}
\end{proof}

This result again suggests that the Ayed-Kuo integral underestimates the expected wealth of the insider.

\section{Insider trading with full information}

We readdress now the problem in~\cite{escudero}.
Contrary to the previous situation, our trader is now an insider who fully trusts her/his information on the future price of the stock,
but s/he is not allowed to borrow any money. Then denote
\begin{eqnarray}\nonumber
d\bar{S}_0 &=& \rho \, \bar{S}_0 \, dt, \\ \nonumber
\bar{S}_0(0) &=& 1,
\end{eqnarray}
and
\begin{eqnarray}\nonumber
d\bar{S}_1 &=& \mu \, \bar{S}_1 \, dt + \sigma \, \bar{S}_1 \, dB_t, \\ \nonumber
\bar{S}_1(0) &=& 1.
\end{eqnarray}
Under these conditions the insider always bets the most profitable asset, that is we consider the initial value problems
\begin{subequations}
\begin{eqnarray}\nonumber
dS_0 &=& \rho \, S_0 \, dt, \\ \nonumber
S_0(0) &=& M \, \1_{\{\bar{S}_1(T) \le \bar{S}_0(T)\}},
\end{eqnarray}
\end{subequations}
and
\begin{eqnarray}\nonumber
d^* S_1 &=& \mu \, S_1 \, dt + \sigma \, S_1 \, d^* B_t, \\ \nonumber
S_1(0) &=& M \, \1_{\{\bar{S}_1(T) > \bar{S}_0(T)\}},
\end{eqnarray}
where $M$ is the total initial wealth of the insider trader. Since this last initial value problem is another stochastic differential equation subject to an
anticipating initial condition, it cannot be interpreted in the sense of It\^o. Both the Hitsuda-Skorokhod and Russo-Vallois senses of it were analyzed
in~\cite{escudero}, so herein, as specified in the notation of the equation, we focus on the Ayed-Kuo sense.
Given the equality of the events
$$
\{\bar{S}_1(T) > \bar{S}_0(T)\} = \{ B(T) > (\rho - \mu + \sigma^2 /2) T/\sigma \}
$$
and
$$
\{\bar{S}_1(T) \leq \bar{S}_0(T)\} = \{ B(T) \leq (\rho - \mu + \sigma^2 /2) T/\sigma \},
$$
we may rewrite this problem as
\begin{subequations}
\begin{eqnarray} \label{ak3}
d^* S_1 &=& \mu \, S_1 \, dt + \sigma \, S_1 \, d^* B_t \\ \label{ak4}
S_1(0) &=& M \1_{\{ B(T) > (\rho - \mu + \sigma^2 /2) T/\sigma \}}.
\end{eqnarray}
\end{subequations}
We conjecture that the solution of $\eqref{ak3}$ and $\eqref{ak4}$ is given by
\begin{equation}\label{ak5}
S_1(t) = M \1_{\{B(T)-\sigma t > (\rho- \mu + \sigma^2/2) T/\sigma\}} \exp\left[(\mu - \sigma^2/2)t+ \sigma B(t)\right].
\end{equation}
Our conjecture is based on the fact that for an initial condition that were a continuous function of Brownian motion the results in~\cite{hksz} would apply;
in fact, this solution would be nothing else but a particular case of~\eqref{sstar}. Although it is uncertain to us whether~\eqref{ak5} is a solution in Ayed-Kuo
sense, we may check that it is a solution to the Hitsuda-Skorokhod initial value problem:
\begin{subequations}
\begin{eqnarray}\label{hs3}
\delta S_1 &=& \mu \, S_1 \, dt + \sigma \, S_1 \, \delta B_t \\ \label{hs4}
S_1(0) &=& M \1_{\{ B(T) > (\rho - \mu + \sigma^2 /2) T/\sigma \}}.
\end{eqnarray}
\end{subequations}
Given the conjectured equivalence between the Ayed-Kuo and Hitsuda-Skorokhod integrals~\cite{kuo2014}, which was checked in a particular case
in Proposition~\ref{prophssol}, this fact further supports the present conjecture.

We begin by recalling the result in~\cite{buckdahn} on the existence and uniqueness of the solution to stochastic differential equations of the form
\begin{equation}\label{buckdahn1}
\delta Y_t = \mu_t \, Y_t \, dt + \sigma_t \, Y_t \, \delta B_t, \quad Y_0=\eta, \qquad 0 \leq t \leq T,
\end{equation}
with $\sigma_t \in L^{\infty}([0,1])$, $\mu_t \in L^{\infty}([0,1]\times \Omega)$, and $ \eta  \in L^p(\Omega)$, $p > 2$.
Then, the unique solution of~\eqref{buckdahn1} is given by the formula
\begin{equation}\label{buckdahn2}
Y_t = \eta(U_{0,t}) \, \exp \left[ \int_0^t \mu_s(U_{s,t}) \, ds \right] X_t,  \ \ \ \ \mbox{a.~s.},  \ \ \ \ 0 \leq t \leq T,
\end{equation}
with the notation
$$
X_t = \exp \left( \int_0^t \sigma_s \, \delta B_s - \frac{1}{2} \int_0^t \sigma_s^2 \, ds \right), \qquad 0 \leq t \leq T,
$$
and, for $0\leq s \leq t \leq T$, the Girsanov transformation
$$
U_{s,t} B(u) = B(u) - \int_0^u \1_{[s,t]}(r) \, \sigma_r \, dr, \ \ \ \ 0 \leq u \leq T.
$$

\begin{theorem}
The unique solution of~\eqref{hs3} and~\eqref{hs4} is~\eqref{ak5}.
\end{theorem}

\begin{proof}
The statement is a particular case of the existence and uniqueness result in~\cite{buckdahn}. Indeed, taking $\sigma_s=\sigma$, $\mu_s=\mu$ as two constant processes and $\eta=M \1_{\{B(T)> (\rho - \mu + \sigma^2/2)T/\sigma\}}$,
the different factors in $\eqref{buckdahn2}$ become
\begin{eqnarray}\nonumber
e^{\mu t} &=& \exp\left[ \int_0^t b_s(U_{s,t}) ds \right], \\ \nonumber
X_t &=& \exp\left[\sigma B(t) - \sigma^2 t /2 \right], \\ \nonumber
\eta(U_{0,t}) &=& M \1_{\{B(T)-t> (\rho - \mu + \sigma^2/2)T/\sigma\}}.
\end{eqnarray}
Since
$$
\mathbb{E} \left( \left|\eta\right|^p \right) = \mathbb{E} (\eta) = P\left[ B(T) > (\rho - \mu + \sigma^2/2)T/\sigma \right] < \infty,
$$
the result follows.
\end{proof}

\section{Conclusions}

In this work we have considered a simplified model of insider trading that allows to directly compare the financial significance of three
different anticipating integrals: the Hitsuda-Skorokhod, the Russo-Vallois, and the Ayed-Kuo stochastic integrals.
In particular, we have established the string of relations:
$$
\mathbb{E}\left[S^{\text{(HS)}}(T)\right] = \mathbb{E}\left[S^{\text{(AK)}}(T)\right] <
\mathbb{E}\left[S^{\text{(I)}}(T)\right] < \mathbb{E}\left[S^{\text{(RV)}}(T)\right],
$$
where $S^{\text{(HS)}}(T)$, $S^{\text{(RV)}}(T)$, and $S^{\text{(AK)}}(T)$ denote respectively the wealth of a Hitsuda-Skorokhod, Russo-Vallois, and
Ayed-Kuo insider at some future time $T$, while $S^{\text{(I)}}(T)$ denotes the corresponding wealth of an honest trader.
This result shows that, contrary to the Russo-Vallois forward integral, the Hitsuda-Skorokhod and the Ayed-Kuo integrals are not well-suited to
model (this sort of) insider trading in finance, at least in our limited setting. This is so to the extent that, while on one hand it always holds that
$$
0 < \mathbb{E}\left[S^{\text{(I)}}(T)\right] < \mathbb{E}\left[S^{\text{(RV)}}(T)\right],
$$
on the other hand
$\mathbb{E}\left[S^{\text{(HS)}}(T)\right]$ and $\mathbb{E}\left[S^{\text{(AK)}}(T)\right]$ may even become negative
(what represents the development of a debt) depending on the parameter values. This result is not present in~\cite{escudero} and it is due to
the possibility of borrowing money that the insider is allowed for in the present work but not in the previous one.
This difference nevertheless highlights the qualitative agreement between the results in the present work and those in~\cite{escudero}.
We recall that the analysis of the Ayed-Kuo integral is new and that it agrees with the Hitsuda-Skorohod integral in this setting, what
leaves the Russo-Vallois forward integral, among these three, as the natural candidate to model insider trading within the formalism
of anticipating stochastic calculus, at least under our present hypotheses. At this point, we would like to highlight that our analysis follows from rather
restrictive assumptions and that it would be of interest to extend the present comparison to more general financial scenarios in order to confirm/correct/broaden these conclusions.

Regarding the {\it noise interpretation dilemma} alluded to in the Introduction, we may highlight an issue that is absent, to the best of our knowledge,
in the classical It\^o vs. Stratonovich problem, but rather requires anticipating stochastic calculus. The solution $\eqref{ak5}$
to the Hitsuda-Skorokhod problem~$\eqref{hs3}$-$\eqref{hs4}$ has discontinuous sample paths with finite probability;
for another example of this see~\cite{buckdahn}.
However, model~$\eqref{hs3}$ and $\eqref{hs4}$ interpreted in the Russo-Valois sense, that is
\begin{subequations}
\begin{eqnarray}\nonumber
d^- S_1 &=& \mu \, S_1 \, dt + \sigma \, S_1 \, d^- B_t \\ \nonumber
S_1(0) &=& M \1_{\{ B(T) > (\rho - \mu + \sigma^2 /2) T/\sigma \}},
\end{eqnarray}
\end{subequations}
has a solution~\cite{escudero}
\begin{equation}\nonumber
S_1(t) = M \1_{\{B(T) > (\rho- \mu + \sigma^2/2) T/\sigma\}} \exp\left[(\mu - \sigma^2/2)t+ \sigma B(t)\right]
\end{equation}
with continuous sample paths almost surely. In other words, noise interpretation in the anticipating setting
can change the continuity properties of the sample paths of the solution. Proving our conjecture on formula~\eqref{ak5} being the
solution to problem~\eqref{ak3}-\eqref{ak4} would show that changing the interpretation from the Russo-Vallois to the Ayed-Kuo one
changes an almost surely $t-$continuous solution to a solution that is $t-$discontinuous with finite probability too.

The mentioned conjecture is the missing step in the triple comparison between these anticipating stochastic integrals carried out
in~\cite{escudero} and herein. Proving or disproving this conjecture is anyway a first step towards the understanding of noise
interpretation in anticipating stochastic calculus. Its role in different financial problems as well as other fields of knowledge
such as physics or chemistry is still to be explored.


\bibliographystyle{amsplain}

\end{document}